\documentclass[11pt]{amsart}

\usepackage[marginratio=1:1]{geometry}
\usepackage{color}
\definecolor{refs}{rgb}{0.7,0,0}
\definecolor{ext}{RGB}{112,112,112}
\definecolor{cite}{RGB}{034,113,179}
\usepackage[colorlinks=true,citecolor=cite,linkcolor=refs,urlcolor=ext]{hyperref}

\newtheorem{theorem}{Theorem}

\newtheorem{corollary}{Corollary}
\newtheorem{proposition}{Proposition}

\theoremstyle{definition}
\newtheorem*{definition}{Definition}
\newtheorem*{remark}{Remark}

\newcommand{\R}{\mathbb{R}}
\newcommand{\N}{\mathbb{N}}

\newcommand{\bS}{\mathbb{S}}

\newcommand{\spn}{\operatorname{span}}

\newcommand{\rk}{\operatorname{rk}}

\newcommand{\arctanh}{\operatorname{arctanh}}

\newcommand{\VV}{\mathcal{V}}
\newcommand{\DD}{\mathcal{D}}

\newcommand{\XX}{\mathcal{X}}
\newcommand{\CC}{\mathcal{C}}

\newcommand{\PP}{\mathbb{P}}

\newcommand{\Hom}{\mathrm{Hom}}

\title[]{Cone structures and path geometries with constant torsion}
\author[W.~Kry\'nski]{Wojciech Kry\'nski}
\address{
Institute of Mathematics, Polish Academy of Sciences, \'Sniadeckich 8, 00-656 Warsaw, Poland}

\email{krynski@impan.pl}

\subjclass[2010]{}
\begin{document}

\begin{abstract}
We study three-dimensional path geometries with nontrivial torsion of maximal rank. We introduce the notion of constant torsion and show that such path geometries are in one-to-one correspondence with certain cone structures modeled on homogeneous ruled surfaces. We also describe these cone structures in terms of solutions to new integrable systems admitting dispersionless, non-isospectral Lax pairs.
\end{abstract}

\maketitle

\section{Introduction}\label{sec:intro}
There is a well-established theory relating 4-dimensional conformal metrics of split signature to 3-dimensional torsion-free path geometries \cite{G}, which can be seen as a manifestation of twistor theory in the real category \cite{CDT}. In this paper, we study path geometries with non-vanishing torsion, which turn out to be naturally associated with certain types of cone structures. These structures generalize conformal metrics by replacing the usual field of light cones with a field of more general, non-quadratic cones.
Building on our earlier results \cite{KM-Cayley,K2}, we characterize and provide a complete classification of path geometries with constant torsion in terms of the associated cone structures. Our main result is as follows:
\begin{theorem}\label{thm2}
There is a one-to-one correspondence between integrable isotypic cone structures in $\R^4$ modeled on projective surfaces defined by equations
\begin{equation}\label{eq:surface1}
\ln(z)=2\arctan\left(\frac{x}{y}\right)
\end{equation}
or
\begin{equation}\label{eq:surface2}
\ln\left(\frac{x^2+y^2}{1+z^2}\right)=2\left(\arctan\left(\frac{x}{y}\right)-\arctan(z)\right)
\end{equation}
and 3-dimensional path geometries whose torsion $T^X$ is of rank-2, satisfies $\nabla T^X\simeq T^X$ and $\hat\bS^X(T^X)=0$ and has two real roots, or two imaginary roots, respectively.
\end{theorem}
The surfaces \eqref{eq:surface1} and \eqref{eq:surface2} are homogeneous, ruled projective surfaces given in \cite[formulas (17) and (20)]{DSV}  (see also \cite[surfaces (2) and (3)]{DDKR}).

The terminology used in the formulation of the theorem will be explained in the next section. For now, let us simply note that an isotypic cone structure refers to a cone structure in which the cone at each point of the manifold is projectively equivalent to a fixed one (in our case, the one given by \eqref{eq:surface1} or \eqref{eq:surface2}, respectively). Moreover, the integrability of a cone structure generalizes the notion of (anti-)self-duality for conformal metrics (see \cite{DW}). Indeed, it is well known that a conformal metric of split signature is self-dual if and only if the $\alpha$-planes can be integrated to yield $\alpha$-submanifolds. A similar notion can be introduced for general cone structures modeled on ruled surfaces.  The operators $\nabla$ and $\hat\bS^X$ are certain differential operators canonically associated with systems of pairs of second-order ordinary differential equations (ODEs), originally introduced in \cite{KM-Cayley}, whose definitions are recalled in Section \ref{sec:prelim}. Note that, locally a path geometry is just a system of second order ODEs given up to point transformations, see e.g. \cite{G}.

Let us point out that the torsion of a system is defined only up to multiplication by a function. Strictly speaking, it can be regarded as a section of $S^2(\XX^*)\otimes\Hom(\VV,\VV)$, where $\XX$ and $\VV$ are certain natural vector bundles of ranks $1$ and $2$, respectively, defined over the space of $1$-jets, and where $S^2(\XX)$ denotes the symmetric tensor product. Given a section $X\in\Gamma(\XX)$, we denote by $T^X$ the corresponding section of $\Hom(\VV,\VV)$. This assignment satisfies the property $T^{fX}=f^2T^X$. Therefore, \emph{a priori}, it is far from obvious what it means for the torsion to be constant. However, there exists a distinguished class of sections of $\XX$, called later \emph{projective vector fields}, on which the Möbius group acts transitively. The conditions $\nabla T^X \simeq T^X$ and $\hat{\mathbb{S}}^X(T^X) = 0$ together imply -- as will be explained later -- that there exists a projective vector field $X$ for which $T^X$ is constant. In this sense, we shall say that the torsion is constant.

Theorem~\ref{thm2} completes the classification of 3-dimensional path geometries (or the pairs of ODEs) with constant torsion, in the above sense, in terms of the associated cone structures. Indeed, the case of torsion of rank~$1$ is extensively studied in~\cite{KM-Cayley}, where the so-called Cayley structures are introduced. These are the second simplest examples of cone structures after conformal structures. The torsion-free case, on the other hand, is a classical result due to Grossman~\cite{G}.
This classification is summarized in the following table:
\begin{table}[h!]
\centering
\begin{tabular}{|c|l|}
\hline
\textbf{Torsion} & \textbf{Cone structures} \\
\hline
$T = 0$ & Quadric \\
\hline
$\operatorname{rank} T = 1$ & Cayley cubic \\
\hline
$\operatorname{rank} T = 2$ and real roots & Surface~\eqref{eq:surface1} \\
\hline
$\operatorname{rank} T = 2$ and complex roots & Surface~\eqref{eq:surface2} \\
\hline
\end{tabular}
\caption{Classification of torsion types and associated surfaces.}
\label{tab:torsion-surfaces}
\end{table}

Let us point out that cone structures, apart from their role in the geometric theory of differential equations \cite{M}, have also attracted attention in the study of the local geometry of complex varieties \cite{H1,H2}.

Another results of the present paper concern integrable systems. It is well known that large classes of dispersionless integrable systems are naturally related to the conformal geometry. In the present paper we present a similar result, however, the cones are no longer quadratic but projectively equivalent to \eqref{eq:surface1} or \eqref{eq:surface2}, respectively.

We prove the following:
\begin{theorem}\label{thm1}
There is a one-to-one correspondence between integrable isotypic cone structures in $\R^4$ modeled on \eqref{eq:surface1} or \eqref{eq:surface2}, respectively, and solutions $(u,v)$ to the dispersionless Lax system
\begin{equation}\label{eq:system}
[L_0,L_1]=0 \mod L_0,L_1,
\end{equation}
where, in the case of surface \eqref{eq:surface1}, the Lax pair is defined as
\begin{equation}\label{eq:Laxpair1}
\begin{aligned}
&L_0=Y_0+ \left((\sinh \lambda)^2Y_0(v_{p^2})-(\sin\lambda)(\sinh
\lambda) Y_1(v_{p^1})\right)\partial_\lambda,\\
&L_1=Y_1+ \left((\sin \lambda)^2Y_1(u_{p^1})-(\sin\lambda)(\sinh
\lambda) Y_0(u_{p^2})\right)\partial_\lambda,
\end{aligned}
\end{equation}
with
\begin{equation}\label{eq:horizontal1}
\begin{aligned}
&Y_0=(\cos \lambda)\partial_{p^1}+(\sin \lambda)(\partial_{x^1}+u_{p^1}\partial_{p^1}+v_{p^1}\partial_{p^2})\\
&Y_1=(\cosh \lambda)\partial_{p^2}+(\sinh \lambda)(\partial_{x^2}+u_{p^2}\partial_{p^1}+v_{p^2}\partial_{p^2}),
\end{aligned}
\end{equation}
and, in the case of surface \eqref{eq:surface2}, the Lax pair is defined as
\begin{equation}\label{eq:Laxpair2}
\begin{aligned}
&L_0=Y_0+ \eta_0(\lambda)\partial_\lambda,\\
&L_1=Y_1+ \eta_1(\lambda)\partial_\lambda,
\end{aligned}
\end{equation}
with
\begin{equation}\label{eq:horizontal2}
\begin{aligned}
Y_0=(\cosh\lambda\cos\lambda)\partial_{p^1}-(\sinh\lambda\sin \lambda)\partial_{p^2}+\frac{1}{2}(\sinh\lambda\cos \lambda+\cosh\lambda\sin \lambda)(\partial_{x^1}+u_{p^1}\partial_{p^1}+v_{p^1}\partial_{p^2})\\
+\frac{1}{2}(\sinh\lambda\cos \lambda-\cosh\lambda\sin \lambda)(\partial_{x^2}+u_{p^2}\partial_{p^1}+v_{p^2}\partial_{p^2})\\
Y_1=(\sinh\lambda\sin \lambda)\partial_{p^1}+(\cosh\lambda\cos \lambda)\partial_{p^2}-\frac{1}{2}(\sinh\lambda\cos \lambda-\cosh\lambda\sin \lambda)(\partial_{x^1}+u_{p^1}\partial_{p^1}+v_{p^1}\partial_{p^2})\\
+\frac{1}{2}(\sinh\lambda\cos \lambda+\cosh\lambda\sin \lambda)(\partial_{x^2}+u_{p^2}\partial_{p^1}+v_{p^2}\partial_{p^2}),
\end{aligned}
\end{equation}
and
\[
\begin{aligned}
\eta_0(\lambda)=
(\sinh\lambda\cos\lambda-\cosh\lambda\sin\lambda)
\left((\sinh\lambda\cos\lambda-\cosh\lambda\sin\lambda)(Y_0(u_{p^1})+Y_1(u_{p^2}))\right.\\
\left.-(\sinh\lambda\cos\lambda+\cosh\lambda\sin\lambda)(Y_0(u_{p^2})-Y_1(u_{p^1}))\right)\\
+(\sinh\lambda\cos\lambda+\cosh\lambda\sin\lambda)
\left((\sinh\lambda\cos\lambda-\cosh\lambda\sin\lambda)(Y_0(v_{p^1})+Y_1(v_{p^2}))\right.\\
\left.+(\sinh\lambda\cos\lambda+\cosh\lambda\sin\lambda)(Y_0(v_{p^2})-Y_1(v_{p^1}))\right)\\
\eta_1(\lambda)=
-(\sinh\lambda\cos\lambda+\cosh\lambda\sin\lambda)
\left((\sinh\lambda\cos\lambda-\cosh\lambda\sin\lambda)(Y_0(u_{p^1})+Y_1(u_{p^2}))\right.\\
\left.-(\sinh\lambda\cos\lambda+\cosh\lambda\sin\lambda)(Y_0(u_{p^2})-Y_1(u_{p^1}))\right)\\
+(\sinh\lambda\cos\lambda-\cosh\lambda\sin\lambda)
\left((\sinh\lambda\cos\lambda-\cosh\lambda\sin\lambda)(Y_0(v_{p^1})+Y_1(v_{p^2}))\right.\\
\left.+(\sinh\lambda\cos\lambda+\cosh\lambda\sin\lambda)(Y_0(v_{p^2})-Y_1(v_{p^1}))\right).
\end{aligned}
\]
\end{theorem}
We show that the solution space of the system depends on at most six functions of three variables. We also describe a subclass of solutions depending on a two functions in 2 variables, thereby proving that the systems admits non-trivial solutions. 

Similar results with non-quadratic cones are known in the literature. This includes $GL(2)$-geometry \cite{FK2,KMet}, as well as the aforementioned Cayley structures \cite{KM-Cayley}. All this examples involve polynomial functions. To the best of our knowledge, systems given in Theorem \ref{thm1} are the first examples of partial differential system in four independent variables with nonisospectral dispersionless Lax pairs involving transcendental, rather than algebraic, dependence on the spectral parameter, cf.\ e.g.\ the discussion in \cite{S} (isospectral Lax pairs involving arbitrary functions can be found in \cite{Mor}; for 3-dimensional systems with transcendental Lax pairs see e.g.\ \cite{FMS}). One should, however, keep in mind that the systems under study are rather heavily overdetermined and, in fact, boil down to systems for functions of three, rather than four, independent variables.

\section{Preliminaries}\label{sec:prelim}
\subsection{Cone structures.} We shall restrict here to the cone structures on 4-dimensional manifolds. Given a manifold $M$, $\dim M=4$, we shall assume that there is an immersed surface $\hat \CC_x$ in the projective tangent space $P(T_xM)$,  for each $x\in M$. It follows that the corresponding preimage $\CC_x\subset T_xM$, $\CC_x=q^{-1}(\hat\CC_x)$, where $q\colon TM\to P(TM)$ is the canonical quotient map, is a 3-dimensional cone in the tangent space. We shall assume that $\CC_x$ depends smoothly on $x$.

\begin{remark}
If all $\CC_x$ are quadratic cones, then the they uniquely determine a class of conformaly equivalent metrics on $M$ of Lorentzian or split signature. Indeed, one can treat the cones as sets of null directions of the metrics. By analogy, more general cone structures in dimension 4 are sometimes referred to as the causal structures (c.f. \cite{KM-Cayley,M}).
\end{remark}

Later we shall need the following two definitions.

\begin{definition}
Let $\CC_V$ be a cone in a vector space $V=\R^4$ and let $\hat\CC_V$ be the corresponding surface in $\PP(V)=\PP^3$. A cone structure $\CC$ on a manifold $M$ is \emph{isotypic}, modeled on $\CC_V$ (or on $\hat\CC_V$), if for any $x\in M$, $\hat \CC_x$ is projectively equivalent to $\hat \CC_V$.
\end{definition}

\begin{definition}
A cone structure $\CC$ on a manifold $M$ is \emph{ruled} if all surfaces $\hat\CC_x\subset P(TM)$, $x\in M$ are ruled.
\end{definition}

If a cone structure is ruled then $\CC_x$ is a union of 2-dimensional subspaces of $T_xM$, we shall refer to them as $\alpha$-planes by analogy to the conformal geometry of split signature, where the quadratic cone has two rulings by so-called $\alpha$- and $\beta$-planes (c.f. \cite{DW}). We shall denote by $N_\CC$ the space of $\alpha$-planes defined by a given ruling of $\CC$. It follows that $N_\CC$ is a $5$-dimensional manifold. It is a fibered bundle over $M$, with a canonical projection denoted $\pi\colon N_\CC\to M$. The bundle $N_\CC$ possesses additional rich geometric structure defined by tautologically defined distributions. Namely, 
\[
\XX(p)=\pi_*^{-1}(0), \qquad \DD(p)=\pi_*^{-1}(p)
\]
where $p\in N_\CC$ is an $\alpha$-plane. In other words, $\XX$ is a rank-1 distribution tangent to fibers of $\pi$ and $\DD$ is the tautological rank-3 distribution, defined as the lift of $\alpha$-planes. Notice that $\XX\subset \DD$.

\begin{definition}
A cone structure is \emph{non-degenerate} if the pair $(\XX,\DD)$ is \emph{regular}, i.e. $[\XX,\DD]=TN_\CC$, where $[\XX,\DD]$ is defined as a distribution spanned by Lie brackets of all sections of $\XX$ and $\DD$.
\end{definition}

One has the following result which can be deduced directly from \cite{Y} specified to systems of second order equations.
\begin{proposition}\label{prop1}
Given a regular pair of distributions $(\XX,\DD)$ on a 5-dimensional manifold $N$, such that $\rk\XX=1$, $\rk\DD=3$, $\XX\subset\DD$ there are local coordinates $t,x^1,x^2,p^1,p^2$ such that 
\[
\XX=\spn\{X_F\},\quad\VV=\spn\{X_F,\partial_{p^1},\partial_{p^2}\}
\]
where
\[
X_F=\partial_t+p^1\partial_{x^1}+p^2\partial_{p^2}+F^1\partial_{x^1}+F^2\partial_{x^2}
\]
for some functions $F^1,F^2\colon N\to \R$ if and only if $\DD$ possesses an integrable co-rank 1 sub-distribution $\VV$ such that $\DD=\VV\oplus\XX$.
\end{proposition}
\begin{proof}
Indeed, if the pair is regular and $\DD$ possesses an integrable co-rank 1 sub-distribution then necessarily $\DD$ satisfies assumptions of of \cite{Y} and consequently it is locally equivalent to the canonical distribution on the space of 1-jets. Additionally, since $\XX$ is not contained in the integrable sub-distribution (since $[\XX,\DD]= TN$) it has to be of the form $X_F$ for some choice of functions $F^1$ and $F^2$.
\end{proof}

The coordinates of Theorem establish local equivalence of $N$ with the space $J^1(\R,\R^2)$ of jets of functions $f\colon \R\to\R^2$. In this context the pair $(\XX,\DD)$, or equivalently $(\XX,\VV)$, encodes a system of second order ODEs. Indeed, $\XX$ is spanned by a total derivative vector field and $\VV$ is tangent to a fiber of a natural projection $J^1(\R,\R^2)\to J^0(\R,\R^2)$.

\begin{definition}
A ruled cone structure $\CC$ is \emph{integrable} if the associated pair $(\XX,\VV)$ on $N_\CC$ is regular and $\DD$ possesses integrable sub-distribution $\VV$.
\end{definition}

\begin{remark}
In \cite{K2}, the pairs $(\XX,\DD)$ satisfying conditions given in Proposition \ref{prop1} are called regular pairs of \emph{equation type}. In particular, any integrable cone structure gives rise to a regular pair of equation type.
\end{remark}

\begin{remark}
Notice that the existence of the integrable sub-distribution $\VV$ in  $\DD$ on $N$ implies that $N$ is equipped with a 2-dimensional foliation. In the case of cone structures the leaves of the foliation are transverse to the fibers of the bundle $N_\CC\to M$ and the projection of each leaf is, by the construction, tangent to $\alpha$-planes. We get 3-dimensional space $T$ of sub-manifolds in $M$ tangent to all $\alpha$-planes, which are called $\alpha$-manifolds by the analogy, again, to the conformal case. This leads to the following double fibration, twistorial-like, picture
\[
M\longleftarrow N_\CC\longrightarrow T
\]
where the twistor space $T$ is the space of leaves of $\VV$, or the space of $\alpha$-manifolds in $M$. Conversely, any such a family of $\alpha$-manifolds gives rise to a foliation of $N_\CC$. By Proposition \ref{prop1} these are exactly the ruled cone structures of equation type and $T$ can be locally identified with the space $J^0(\R,\R^2)$ of $0$-jests. 
\end{remark}

\subsection{Systems of ODEs.}
We shall concentrate on systems of two second-order ODEs, given in the form
\[
\ddot x^i=F^i(t,x,\dot x), \quad i=1,2.
\]
Systems of this form are encoded by pairs $(\XX, \VV)$ described in Proposition~\ref{prop1}. Explicitly,
\[
\XX=\spn\{X_F\},\quad\VV=\spn\{\partial_{p^1},\partial_{p^2}\}
\]
where, as before, $X_F=\partial_t+p^1\partial_{x^1}+p^2\partial_{p^2}+F^1\partial_{x^1}+F^2\partial_{x^2}$.
The general equivalence problem for systems of second-order ODEs was solved in \cite{F}, where a Cartan connection is assigned to the system. The invariants are of two kinds: torsion and curvature. In coordinates, they are given by the following formulas
\begin{equation}\label{eq:formulaT}
 T^i_j=F^i_j-\textstyle{\frac{1}{2}\delta^i_jF^k_k},\qquad\qquad R^i_{jkl}=F^i_{jkl}-\textstyle{\frac{3}{4}}F^r_{r(jk}\delta^i_{l)}
\end{equation}
where $1\leq i,j,k,l\leq 2$ and 
\begin{equation}\label{eq:formulaR}
F^i_{j}=\textstyle{-\partial_{x^j} F^i+\frac{1}{2}X_F(\partial_{p^j} F^i)-\frac{1}{4}\partial_{p^k} F^i\partial_{p^j} F^k,}\quad F^i_{jkl}=\partial_{p^j}\partial_{p^k}\partial_{p^l}F^i.
\end{equation}
It is important to stress that both $T^i_j$ and $R^i_{jkl}$ are defined only up to multiplication by a function. In fact, $T = (T^i_j)$ is a section of $S^2(\XX^*) \otimes\Hom(\VV,\VV)$, where $S^2$ denotes the bundle of symmetric 2-tensors. We will denote by $T^X$ the evaluation of $T$ on a section $X \in \Gamma(\XX)$.

\begin{remark}
The torsion $T^X$ can be interpreted in terms of sections of $\XX$ and $\VV$. Namely, we say that a section $X$ of $\XX$ is \emph{projective}, and a frame $(V_1, V_2)$ of $\VV$ is \emph{normal} (cf. \cite{K2}), if
\begin{equation}\label{eq:torsion}
[X,[X,V_i]]=T_i^jV_j\mod \XX,
\end{equation}
for certain coefficients $T^i_j$ such that $T^1_1 + T^2_2 = 0$. It turns out that these coefficients are the components of the torsion in the basis $(V_1,V_2)$, i.e. $T^i_j$ in \eqref{eq:torsion} are linear combinations of $T^i_j$ in \eqref{eq:formulaT}. As a matter of fact, formula \eqref{eq:torsion} can be taken as an equivalent definition of the torsion. If $X$ and $fX$ are two projective vector fields, then the transformation rule for the corresponding torsion coefficients is $T^{fX} \mapsto f^2 T^X$, which justifies the aforementioned interpretation of $T$ as a section of $S^2(\XX^*) \otimes \operatorname{Hom}(\VV, \VV)$. Projective vector fields are only defined up to the action of Möbius transformations
\[
t\mapsto \frac{at+b}{ct+d}
\]
which act on the parameterizations of projective vector fields. Consequently, the projective vector fields form a 2-parameter family of vector fields along an integral curve of $\XX$. Indeed, it is directly verified (see e.g. \cite{K2} for details) that two vector fields $X$ and $fX$ are projective if and only if
\begin{equation}\label{eq:schwarzian}
fX^2(f)-2X(f)^2=0.
\end{equation}
Moreover, when rewritten in terms of parameterizations of integral curves of $X$ and $fX$, equation \eqref{eq:schwarzian} is equivalent to the vanishing of the Schwarzian derivative of a natural parameter of $fX$ with respect to a natural parameter of $X$.
\end{remark}

Later, in Theorem~\ref{thm2}, we shall deal with systems with nonvanishing torsion. We use two differential operators, $\nabla$ and $\hat{\bS}$, whose definitions are given in \cite{KM-Cayley} (see also \cite{K2} for a more detailed exposition). The first is an operator
\[
\nabla\colon \Gamma(\Hom(\VV,\VV))\to\XX^*\otimes \Gamma(\Hom(\VV,\VV)),
\]
which can be interpreted as a partial connection on the bundle of endomorphisms of the vector bundle $\VV$, acting in the direction of $\XX$. If a normal frame is fixed and $A\in \Gamma(\Hom(\VV,\VV))$ is represented by a matrix $A^i_j$ in this frame, then $\nabla_XA$ is represented by the matrix $X(A^i_j)$. Later on, we shall apply $\nabla$ to the torsion. Although the torsion is only defined up to multiplication by a positive function, the condition $\nabla T^X = \phi T^X$ is well defined, where $\phi$ is a smooth function.

The second is a second-order operator which, for any projective vector field $X\in\Gamma(\XX)$, is a mapping
\[
\hat \bS^X\colon \Gamma(\Hom(\VV,\VV))\to \Gamma(\Hom(\VV,\VV)\otimes \Hom(\VV,\VV)),
\]
defined as
\[
\hat \bS^X(A)=\frac{1}{2}(\nabla^2_X A\otimes A+A\otimes\nabla^2_X A)-\frac{5}{4}\nabla_X A\otimes\nabla_X A.
\]
It is proven in \cite{KM-Cayley} (see also \cite{K2} for a more general exposition) that $\hat\bS^X(T^X)$ is a well defined section of $S^6(\XX^*)\otimes\Hom(\VV,\VV)\otimes\Hom(\VV,\VV)$.

\begin{proposition}\label{prop2}
If  $\nabla_X T^X = \phi T^X$ for some function $\phi$ and $\hat{\bS}^X(T^X) = 0$ then there exists a normal frame of $\VV$ and a projective vector field of $\XX$ such that the corresponding torsion matrix is constant.
\end{proposition}
\begin{proof}
Let $X$ be a projective vector field and let $(V_1,V_2)$ be a normal basis, so that equation \eqref{eq:torsion} holds. Then, for any section $A\in\Gamma(\Hom(\VV,\VV))$ with the matrix $(A^i_j)$ in the basis $(V_1,V_1)$ we have $\nabla_X(A)$ is a section of $\Hom(\VV,\VV)$ with the matrix $X(A^i_j)$ in the same basis $(V_1,V_2)$. Consequently, $\nabla_X T^X=\phi T^X$ implies that $X(T^i_j)=\phi T^i_j$, where $T^i_j$ are torsion coefficients of $T^X$ in the basis $(V_1,V_2)$. It follows that in the basis $(V_1,V_2)$, $T^X=(T^i_j)$ is constant up to a multiplication by some function, i.e.
\[
T^X=\psi \tilde T
\]
where $\tilde T$ satisfies $X(\tilde T)=0$ in the basis of $(V_1,V_2)$. Consequently, $\nabla_X\tilde T=0$ and we get that
\[
\hat \bS^X(T^X)=\hat \bS^X(\psi\tilde T^X)=\left(\psi X^2(\psi)-\frac{5}{4}X(\psi)^2\right) \tilde T\otimes\tilde T.
\]
On the other hand, by assumption we have $\hat \bS^X(T^X)=0$. Hence, we get $\psi X^2(\psi)-\frac{5}{4}X(\psi)^2=0$. This is equivalent to equation \eqref{eq:schwarzian} with $f=\sqrt{\psi^{-1}}$, as can be directly verified. Furthermore, the equation  is equivalent to the fact that parameterizations of $X$ and $fX$ are related by a Möbius transformation, which means that $fX$ is also a projective vector field. Finally, the transformation rule $T^{fX}=f^2T^X=f^2\psi\tilde T=\tilde T$ proves that $T^{fX}$ is constant.
\end{proof}

\section{Results}\label{sec:results}
We shall assume that a cone structure $\CC$ on a manifold $M$ is modeled on the surface defined either by \eqref{eq:surface1} or \eqref{eq:surface2}. The cones in the tangent bundle are equivalently described by the following equations
\begin{equation}\label{eq:cone1}
\arctan\left(\frac{x^3}{x^1}\right)=\arctanh\left(\frac{x^4}{x^2}\right),
\end{equation}
or, respectively,
\begin{equation}\label{eq:cone2}
\ln\left(\frac{(x^1)^2+(x^2)^2}{(x^3)^2+(x^4)^2}\right)=2\arctan\left(\frac{x^1x^4+x^2x^3}{x^2x^4-x^1x^3}\right)
\end{equation}
where $(x^1,\ldots,x^4)$ are linear coordinates on $\R^4$. Indeed, in the first case we assign $(x:y:z:1)=(x^2:x^1:x^4:x^3)$ to get equivalence of \eqref{eq:surface1} and \eqref{eq:cone1}, and in the second case we assign $(x:y:z:1)=(x^2:x^1:-x^4:x^3)$ to get equivalence of \eqref{eq:surface2} and \eqref{eq:cone2}.

\begin{proposition}\label{prop2}
The cones defined by equations \eqref{eq:cone1} and \eqref{eq:cone2} are ruled by 2-dimensional planes in $\R^4$, explicitly given by
\begin{equation}\label{eq:param1}
\lambda\mapsto \{(a\cos \lambda,b\cosh \lambda,a\sin \lambda,b\sinh \lambda)\ |\ a,b\in\R\},
\end{equation}
in the case of \eqref{eq:cone1}, and 
\begin{equation}\label{eq:param2}
\begin{aligned}
\lambda\mapsto \{&(ae^\lambda\cos \lambda+be^\lambda\sin\lambda,ae^\lambda\sin\lambda-be^\lambda\cos \lambda,\\
&ae^{-\lambda}\cos \lambda-be^{-\lambda}\sin\lambda,ae^{-\lambda}\sin\lambda+be^{-\lambda}\cos \lambda)\ |\ a,b\in\R\},
\end{aligned}
\end{equation}
in the case of \eqref{eq:cone2}.
\end{proposition}
\begin{proof} 
Checked by direct substitution of \eqref{eq:param1} and \eqref{eq:param2} into \eqref{eq:cone1} and \eqref{eq:cone2}, respectively.
\end{proof}
Notice that
\[
x(t)=c_1\cos t+ c_2\sin t, \quad y(t)=c_3\cosh t+c_4\sinh t,
\]
is the general solutions to the system
\begin{equation}\label{eq:linear1}
\ddot x = -x,\qquad \ddot y=y,
\end{equation}
which has constant diagonal torsion $T=\mathrm{diag}(-1,1)$. Similarly,
\[
x(t)=c_1e^t\cos t+c_2e^t\sin t+c_3e^{-t}\cos t+c_4e^{-t}\sin t,\quad y(t)=\frac{1}{2}\ddot x(t)
\]
is the general solution to the system 
\begin{equation}\label{eq:linear2}
\ddot x = -2y,\qquad \ddot y=2x,
\end{equation}
which has constant anti-diagonal torsion with complex eigenvalues. 

\subsection{Proof of Theorem \ref{thm2}}
According to Proposition \ref{prop1}, an integrable cone structure gives rise to a system of ODEs, with a local identification of $N_\CC$ and $J^1(\R,\R^2)$. Moreover, since we assume that $\CC$ is isotypic and modeled on \eqref{eq:cone1}, or \eqref{eq:cone2}, one can choose a basis $(e_1(x), \ldots, e_4(x))$ in each tangent space $T_xM$ such that the cone $\CC_x$ takes the form \eqref{eq:cone1} or \eqref{eq:cone2}, respectively, in this basis. Then the corresponding parameterizations \eqref{eq:param1} or \eqref{eq:param2} yield that distributions $\XX$ and $\DD$ in Proposition \ref{prop1} can be taken as $\XX=\spn\{\partial_\lambda\}$ and $\DD=\spn\{V_1(\lambda,x),V_2(\lambda,x),\partial_\lambda\}$ where
\begin{equation}\label{eq:frame1}
\begin{aligned}
&V_1(\lambda,x)=(\cos \lambda)e_1(x)+(\sin \lambda)e_3(x),\\
&V_2(\lambda,x)= (\cosh \lambda)e_2(x)+(\sinh \lambda)e_4(x),
\end{aligned}
\end{equation}
in the first case, and
\begin{equation}\label{eq:frame2}
\begin{aligned}
&V_1(\lambda,x)=(e^\lambda\cos\lambda)e_1(x)+(e^\lambda\sin\lambda)e_2(x)+(e^{-\lambda}\cos\lambda)e_3(x)+(e^{-\lambda}\sin\lambda)e_4(x)\\
&V_2(\lambda,x)=(e^\lambda\sin\lambda)e_1(x)-(e^\lambda\cos \lambda)e_2(x)-(e^{-\lambda}\sin\lambda)e_3(x)+(e^{-\lambda}\cos\lambda)e_4(x),
\end{aligned}
\end{equation}
in the second case. Additionally, equation \eqref{eq:torsion} implies that modulo the vertical direction $\partial_\lambda$, the vector fields $V_1$ and $V_2$ constitute a normal frame of $\VV$, corresponding to the projective vector field $\partial_\lambda$. In this frame, the torsion has the diagonal, or anti-diagonal form coinciding with the one for linear systems \eqref{eq:linear1} and \eqref{eq:linear2}, respectively (because coefficients of $e_i$ in $V_1(\lambda,x)$ and $V_2(\lambda,x)$ coincide with the general solutions to the linear systems). 

Conversely, given an ODE satisfying the assumptions, one can, by Proposition~\ref{prop2}, choose a normal frame and a projective vector field satisfying \eqref{eq:torsion} with the appropriate torsion. Solving \eqref{eq:torsion} for $V_1$ and $V_2$ with the prescribed constant torsion yields \eqref{eq:frame1} and \eqref{eq:frame2} as the resulting frame, where now $\lambda$ is a  parameterization of $\XX$ corresponding to the chosen projective vector field. It follows, that the projection of $\VV$ to the solution space of the equation (i.e. the quotient space $J^1(\R, \R^2)/\XX$) defines a cone described by equation \eqref{eq:cone1} or \eqref{eq:cone2}, respectively, in the corresponding tangent space.

\subsection{Proof of Theorem \ref{thm1}}
We already know that any integrable cone structure modeled on \eqref{eq:surface1}, or \eqref{eq:surface2} has a local description in terms of ODEs with the appropriate constant torsion. Moreover, as in the proof of Theorem \ref{thm2}, we identify  $N_\CC$ and $J^1(\R,\R^2)$. Under this identification, the slice $\{t=0\} \subset J^2(\R, \R^2)$ is diffeomorphic to $M$ and this viewpoint allows us to introduce local coordinates $(x^1, x^2, p^1, p^2)$ on $M$. We complement these coordinates with $\lambda$ to get coordinates on $N_\CC$. Note that, although due to Proposition~\ref{prop1}, $N_\CC$ is locally diffeomorphic to $J^1(\R,\R^2)$, the coordinate $\lambda$ is different than $t$. We only have $\XX=\spn\{\partial_\lambda\}=\spn\{X_F\}$ and we can assume that the slices $\{\lambda=0\}$ and $\{t=0\}$ coincide.

Now, as in \cite[Theorem 7.1]{KM-Cayley}, we can pick a normal frame $(V_1,V_2)$ of $\VV$ on $J^1(\R,\R^2)$ corresponding to a projective vector field $X$ such that on the slice $\{t=0\}$, $V_i=\partial_{p^i}$ and $[X,V_i]=\partial_{x^i}+\frac{1}{2}(\partial_{p^i}F^1\partial_{p^1}+\partial_{p^i}F^2\partial_{p^2})$. Therefore, denoting $u=\frac{1}{2}F^1|_{t=0}$ and $v=\frac{1}{2}F^2|_{t=0}$, we know that there is a frame of $\VV$ on $N_\CC$, such that for $\lambda=0$,
\begin{equation}\label{eq:initial}
V_i=\partial_{p^i},\qquad\frac{d}{d\lambda}V_i=\partial_{x^i}+u_{p^i}\partial_{p^1}+v_{p^i}\partial_{p^2}\mod\partial_\lambda.
\end{equation}
Notice that the normal frame \eqref{eq:frame1} satisfies the initial condition \eqref{eq:initial} at $\lambda=0$, with $e_1=\partial_{p^1}$, $e_2=\partial_{p^2}$, $e_3=\partial_{x^1}+u_{p^1}\partial_{p^1}+v_{p^1}\partial_{p^2}$ and $e_4=\partial_{x^2}+u_{p^2}\partial_{p^1}+v_{p^2}\partial_{p^2}$. In the case of \eqref{eq:frame2}, the initial condition is not satisfied and the frame needs to be adjusted. In fact, one can take the following new frame
\begin{equation}\label{eq:frame2bis}
\begin{aligned}
Y_0=(\cosh\lambda\cos\lambda)e_1-(\sinh\lambda\sin \lambda)e_2+\frac{1}{2}(\sinh\lambda\cos \lambda+\cosh\lambda\sin \lambda)e_3\\+\frac{1}{2}(\sinh\lambda\cos \lambda-\cosh\lambda\sin \lambda)e_4\\
Y_1=(\sinh\lambda\sin \lambda)e_1+(\cosh\lambda\cos \lambda)e_2-\frac{1}{2}(\sinh\lambda\cos \lambda-\cosh\lambda\sin \lambda)e_3\\+\frac{1}{2}(\sinh\lambda\cos \lambda+\cosh\lambda\sin \lambda)e_4,
\end{aligned}
\end{equation}
for which \eqref{eq:initial} is satisfied  with $e_1=\partial_{p^1}$, $e_2=\partial_{p^2}$, $e_3=\partial_{x^1}+u_{p^1}\partial_{p^1}+v_{p^1}\partial_{p^2}$ and $e_4=\partial_{x^2}+u_{p^2}\partial_{p^1}+v_{p^2}\partial_{p^2}$ as in the previous case.

We get that the regular pair $(\XX,\DD)$ of the cone structure is spanned by $\partial_\lambda$ and $Y_0$ and $Y_1$ defined by formulas \eqref{eq:horizontal1} or \eqref{eq:horizontal2}, respectively. The Lax pair $L_0$ and $L_1$ is then given by the unique extension of $Y_0$ and $Y_1$ such that $[L_0,L_1]\subset\spn\{Y_0,Y_1,\partial_\lambda\}$ on $N_\CC$. Indeed, since $\DD$ is a $(3,5)$-distribution, the choice of $\VV\subset\DD$ such that $[\VV,\VV]\subset\DD$ is unique and can be found by solving an algebraic system (any distribution with growth vector $(3,5)$ possesses unique rank-2 subdistribution, whose Lie square is contained in the original rank-3 distribution; in our case this is $\VV$). We find  $\VV$ explicitly explicitly by expressing it as $\VV=\spn\{L_0,L_1\}$, where  
\[
L_i = Y_i + \eta_i \partial_\lambda,
\]  
and then computing
\[
[L_0, L_1] = [Y_0, Y_1] + \eta_0\, \partial_\lambda Y_1 - \eta_1\, \partial_\lambda Y_0 \mod \partial_\lambda.
\]  
The right-hand side can be expressed in the basis $(Y_0, Y_2, \partial_{p^1}, \partial_{p^2})$, and we impose the condition that the coefficients of $\partial_{p^1}$ and $\partial_{p^2}$ vanish. This yields a system of linear equations for $\eta_0$ and $\eta_1$, which determines them uniquely (the explicit formulas are given in the statement of the theorem).
 Condition $[L_0,L_1]=0\mod L_0,L_1$ is then equivalent to the integrability of $\VV$, which is equivalent to integrability of the cone structure.

\subsection{Explicit formulas for the Lax system}\label{sec:explicit}
In this section, we provide explicit formulas here for the Lax system \eqref{eq:system} with the Lax pair \eqref{eq:Laxpair1}. We do not provide analogous formulas for the Lax pair \eqref{eq:Laxpair2} due to its higher complexity. However, the two surfaces \eqref{eq:surface1} and \eqref{eq:surface2} are equivalent in the complex category (cf. \cite{DDKR}), as are the complexified corresponding linear systems \eqref{eq:linear1} and \eqref{eq:linear2}. Therefore, the solution space of the PDE system for \eqref{eq:Laxpair2} should be parameterized by the same number of functions as that for \eqref{eq:Laxpair1}.

In order to obtain explicit equations for $u$ and $v$ that are equivalent to system \eqref{eq:system}, we expand the Lie bracket $[L_0, L_1]$ and examine the coefficients in the expansion with respect to the functions $\cos^i(\lambda)\cosh^j(\lambda)\sin^k(\lambda)\sinh^l(\lambda)$, where $i,j \in {0,1}$ and $k,l \in\N$. These functions are independent over $\R$ as functions of $\lambda$, and consequently, the corresponding coefficients must vanish for \eqref{eq:system} to be satisfied. As a result, we obtain the following equations:
\begin{equation}\label{eq:system-explicit}
\begin{array}{l}
u_{p^1p^1p^2}=0,\\ -v_{p^1p^2p^2}=0,\\ -u_{p^1p^1p^2} + v_{p^1p^2p^2}=0,\\ -4 u_{p^1p^2} v_{p^1p^2} + 2 r s=0,\\ 
-4 u_{p^1p^2} v_{p^1p^2} - 2 r s=0,\\ 
-4 r u_{p^1p^2} + 4 s v_{p^1p^2}=0,\\ 
-v_{x^2p^1p^2} - u_{p^2}  v_{p^1p^1p^2} - v_{p^2}  v_{p^1p^2p^2} - 2 u_{p^1p^2} v_{p^1p^2}=0,\\ 
u_{x^1p^1p^2} + u_{p^1}  u_{p^1p^1p^2} + v_{p^1}  u_{p^1p^2p^2} + 2 u_{p^1p^2} v_{p^1p^2}=0,\\ 
u_{p^1p^1p^2} + 3 r u_{p^1p^2} - s v_{p^1p^2} - u_{p^1}  s_{p^1}  - v_{p^1}  s_{p^2}  - s_{x^1}=0,\\ 
v_{p^1p^2p^2} + r u_{p^1p^2} - 3 s v_{p^1p^2} + u_{p^2}  r_{p^1}  + v_{p^2}  r_{p^2}  + r_{x^2}=0,\\ 
-u_{x^1p^1p^2} - u_{p^1}  u_{p^1p^1p^2} - v_{p^1}  u_{p^1p^2p^2} - 6 u_{p^1p^2} v_{p^1p^2} - s_{p^1}=0,\\ 
v_{x^2p^1p^2} + u_{p^2}  v_{p^1p^1p^2} + v_{p^2}  v_{p^1p^2p^2} + 6 u_{p^1p^2} v_{p^1p^2} + r_{p^2}=0,
\end{array}
\end{equation}
where
\begin{equation}\label{eq:rs}
\begin{array}{l}
r = v_{x^2p^1} - v_{x^1p^2}  + u_{p^2} v_{p^1p^1}+ (v_{p^2}-u_{p^1}) v_{p^1p^2} - v_{p^1} v_{p^2p^2}\\
s =u_{x^1p^2} - u_{x^2p^1} - u_{p^2} u_{p^1p^1} +  (u_{p^1} - v_{p^2}) u_{p^1p^2} + v_{p^1} u_{p^2p^2}
\end{array}
\end{equation}
Let us point out that the third of the above equations in \eqref{eq:system-explicit} is an obvious consequence of the first two. Additionally, the fourth and fifth equations in \eqref{eq:system-explicit} are equivalent to the following pair of equations: 
$rs=0$ and $u_{p^1p^2} v_{p^1p^2}=0$. With this in mind, and after some obvious simplifications, it is readily seen that \eqref{eq:system-explicit} is actually equivalent to a simpler system
\begin{equation}\label{eq:sys-new}
\begin{array}{l}
u_{p^1p^1p^2}=0,\\ v_{p^1p^2p^2}=0,\\ u_{p^1p^2} v_{p^1p^2}=0,\\
r s=0,\\ 
r u_{p^1p^2} - s v_{p^1p^2}=0,\\ 
v_{x^2p^1p^2} + u_{p^2}  v_{p^1p^1p^2}=0,\\ 
u_{x^1p^1p^2}  + v_{p^1}  u_{p^1p^2p^2}=0,\\ 
2 r u_{p^1p^2}  - v_{p^1}  s_{p^2}  - s_{x^1}=0,\\ 
2 s v_{p^1p^2} - u_{p^2}  r_{p^1}  - r_{x^2}=0,\\ 
s_{p^1}=0,\\ 
r_{p^2}=0.
\end{array}
\end{equation}

\begin{corollary}
Any solution to \eqref{eq:sys-new} can be written in the following form
\[
\begin{array}{l}
u(x^1,x^2,p^1,p^2)=u^1(x^1,x^2,p^2)+p^1u^2(x^1,x^2,p^2)+u^3(x^1,x^2,p^1),\\
v(x^1,x^2,p^1,p^2)=v^1(x^1,x^2,p^1)+p^2v^2(x^1,x^2,p^1)+v^3(x^1,x^2,p^2),
\end{array}
\]
for some functions $u^i,v^i\colon\R^3\to\R$, $i=1,2,3$.
\end{corollary}

Subsequent analysis of \eqref{eq:sys-new} naturally splits into three cases:

1) $u_{p^1p^2}=0, v_{p^1p^2}\neq 0$

2) $u_{p^1p^2}\neq 0, v_{p^1p^2}=0$

3) $u_{p^1p^2}=0, v_{p^1p^2}=0$

\bigskip

For case 1, we get
\begin{equation}\label{eq:sys-new-1-a}
\begin{array}{l}
v_{p^1p^2p^2}=0,\\ u_{p^1p^2}=0,\\
s=0,\\ 
v_{x^2p^1p^2} +u_{p^2}  v_{p^1p^1p^2}=0,\\ 
 u_{p^2}  r_{p^1}  + r_{x^2}=0,\\ 
r_{p^2}=0.
\end{array}
\end{equation}

For case 2, we get
\begin{equation}\label{eq:sys-new-2-a}
\begin{array}{l}
u_{p^1p^1p^2}=0,\\ v_{p^1p^2}=0,\\
r=0,\\ 
u_{x^1p^1p^2}  + v_{p^1}  u_{p^1p^2p^2}=0,\\ 
v_{p^1}  s_{p^2}  + s_{x^1}=0,\\ 
s_{p^1}=0,\\ 
\end{array}
\end{equation}

For case 3, we get 
\begin{equation}\label{eq:sys-new-3-a}
\begin{array}{l}
u_{p^1p^2}=0,\\ v_{p^1p^2}=0,\\ 
r s=0,\\ 
v_{p^1}  s_{p^2}  + s_{x^1}=0,\\ 
u_{p^2}  r_{p^1}+ r_{x^2}=0,\\ 
s_{p^1}=0,\\ 
r_{p^2}=0,
\end{array}
\end{equation}

\begin{corollary}
Arbitrary functions $u=u(x^1,p^1)$ and $v=v(x^2,p^2)$ solve \eqref{eq:sys-new-3-a}. Moreover, if $u_{p^1p^1p^1}\neq 0$ or $v_{p^2p^2p^2}\neq 0$ then the corresponding cone structure is not diffeomorphic to the flat one (corresponding to the one with $u=v=0$ and defined by the linear system \eqref{eq:linear1}).
\end{corollary}
\begin{proof}
The first assertion follows from direct inspection. The second assertion is a consequence of the fact that, by construction, $u_{p^1p^1p^1}$ and $v_{p^2p^2p^2}$, up to a constant, coincide with the coefficients of the curvature terms $F^1_{111}$ and $F^2_{222}$, respectively, of the corresponding system of ODEs computed on the slice $\{t = 0\}$ in $J^1(\R,\R^2)$. Consequently, if they do not vanish then the system is not linear.

\end{proof}

\subsection*{Acknowledgments}
The author would like to thank Artur Sergyeyev for his valuable comments and for assistance with the computations presented in Section \ref{sec:explicit}, in particular for providing the explicit form of the system and its further reductions in the case of the Lax pair \eqref{eq:Laxpair1}. I am also indebted to the Mathematical Institute of the Silesian University in Opava for its hospitality during my visit in June 2025, when part of the research for the present article was completed.

\end{document}